\documentclass[10pt]{scrartcl}
\title{Convergence Rates for Exponentially Ill-Posed Inverse Problems with Impulsive Noise}
\author{Claudia K\"onig\footnotemark[2],  Frank Werner\footnotemark[3] and Thorsten Hohage\footnotemark[4]}

\usepackage[english] {babel}
\usepackage[utf8]{inputenc}
\usepackage{amssymb}
\usepackage{amsxtra}
\usepackage{amsthm}
\usepackage{a4wide}
\usepackage{mathrsfs}
\usepackage{exscale}
\usepackage{cite}

\DeclareMathOperator*{\argmin}{argmin}

\newcommand{\X}{\mathcal X}

\newcommand{\Op}{F}
\newcommand{\OpL}{T}

\newcommand{\sol}{f}
\newcommand{\udag}{\sol^\dagger}
\newcommand{\data}{g}
\newcommand{\gdag}{\data^\dagger}
\newcommand{\gobs}{\data^{\rm obs}}
\newcommand{\ualdel}{\widehat{\sol}_\alpha}

\newcommand{\manifold}{\mathbb{M}}
\newcommand{\mc}{\mathbb{P}}
\newcommand{\mi}{\manifold\setminus\mc}
\newcommand{\R}{\mathcal R}
\newcommand{\A}{\mathcal A}
\newcommand{\E}{\mathcal E}
\newcommand{\Z}{\mathcal Z}

\newcommand{\Borel}{\mathfrak{B}}

\newcommand{\breg}[1]{\mathcal D\left(#1,\udag\right)}
\newcommand{\Legr}{p}
\newcommand{\SHProj}{P}
\newcommand{\calU}{\mathcal{U}}
\newcommand{\D}{\,\mathrm{d}}

\newcommand{\Rset}{\mathbb R}
\newcommand{\Cset}{\mathbb C}
\newcommand{\Nset}{\mathbb N}
\newcommand{\Zset}{\mathbb Z}
\renewcommand{\SS}{\mathbb{S}}

\begin{document}

\newtheorem{thm}{Theorem}[section]
\newtheorem{prop}[thm]{Proposition}
\newtheorem{lem}[thm]{Lemma}
\newtheorem{cor}[thm]{Corollary}
\newtheorem{rem}[thm]{Remark}
\newtheorem{ex}[thm]{Example}
\newtheorem{ass}{Assumption}
\newtheorem{Def}[thm]{Definition}

\maketitle

\renewcommand{\thefootnote}{\fnsymbol{footnote}}

\footnotetext[2]{Institute for Mathematical Stochastics, University of G\"ottingen, Germany and Inverse Problems in Biophysics Group, Max Planck Institute for Biophysical Chemistry, G\"ottingen, Germany (claudia-juliane.koenig@mathematik.uni-goettingen.de)}
\footnotetext[3]{Felix Bernstein Institute for Mathematical Statistics in the Biosciences, University of G\"ottingen, Germany and Inverse Problems in Biophysics Group, Max Planck Institute for Biophysical Chemistry, G\"ottingen, Germany (Frank.Werner@mpibpc.mpg.de)}
\footnotetext[4]{Institute for Numerical and Applied Mathematics, University of G\"ottingen, Germany (hohage@math.uni-goettingen.de)}

\renewcommand{\thefootnote}{\arabic{footnote}}

\begin{abstract}
This paper is concerned with exponentially ill-posed operator equations with additive impulsive noise on the right hand side,  
i.e.\ the noise is large on a small part of the domain and small or zero outside. 
It is well known that Tikhonov regularization with an $L^1$ data fidelity term outperforms Tikhonov regularization with an $L^2$ fidelity term in this case.  
This effect has recently been explained and quantified for the case of finitely smoothing operators. 
Here we extend this analysis to the case of infinitely smoothing forward operators under standard Sobolev 
smoothness assumptions on the solution, i.e.\ exponentially ill-posed inverse problems. 
It turns out that high order polynomial rates of convergence in the size of the support of large noise 
can be achieved 
rather than the poor logarithmic convergence rates typical for exponentially ill-posed problems. 
The main tools of our analysis are Banach spaces of analytic functions and interpolation-type 
inequalities for such spaces. 
We discuss two examples, the (periodic) backwards heat equation and an inverse problem in gradiometry. 
\end{abstract}
\pagestyle{myheadings}
\thispagestyle{plain}
\markboth{C.~K\"onig,  F.~Werner and T.~Hohage}{Convergence Rates for Exponentially Ill-Posed Inverse Problems with Impulsive Noise}


{\it AMS classification numbers: } 65J20, 65K10, 65J22, 46B70

\section{Introduction}

In this work we analyze Tikhonov-type regularization for exponentially ill-posed problems where the data are corrupted by impulsive noise.
We suppose that the measurements are described by functions $\gobs = \gdag + \xi \in L^1\left(\manifold\right)$ on a submanifold $\manifold \subset \Rset^d$ where $\gdag$ denotes the exact data and $\xi$ the noise function.
The noise $\xi$ is called impulsive if $\left|\xi\right|$ is large on a small part of $\manifold$ and small or even zero elsewhere. Impulsive noise naturally occurs in digital image acquisition due to faulty memory locations or any kind of physical measurements with malfunctioning receivers. Inverse problems with impulsive noise have been studied extensively in the literature (see e.g.\  \cite{chn04,cjk10a,cj12} and the references therein), and several authors \cite{bo04,fh10,g10b,cj12,cjk10b,wh12} also analyzed Tikhonov-type regularization where reconstructions are defined as minimizers of a generalized Tikhonov functional
\begin{equation}\label{eq:tik}
\ualdel \in \argmin\limits_{f \in D\left(F\right)} \left[\frac{1}{\alpha r} \left\Vert F\left(f\right) - \gobs\right\Vert_{L^r \left(\manifold\right)}^r + \R\left(f\right)\right].
\end{equation}
Here $\X$ is a Banach space, $F : D(F)\subset \X \to L^1 \left(\manifold\right)$ the forward operator, $\R : \X \to \left(-\infty, \infty\right]$ a convex penalty term stabilizing the reconstructions, $r \geq 1$, and $\alpha>0$ is a regularization parameter. Most common examples for $\R$ are powers of Banach space norms $\R\left(f\right) = \frac1q \left\Vert f-f_0\right\Vert_{\X}^q$ with $f_0 \in \X$ and $q \geq 1$ or total variation-type functionals. 

It is well known from several numerical experiments \cite{cjk10b,js12,j11} that the choice $r = 1$ in \eqref{eq:tik} leads to much better reconstructions than $r = 2$, and several efficient algorithms have been developed to minimize \eqref{eq:tik} with $r = 1$ (see \cite{kkm05,lsds11,n02,n04,ygo07,yzy09}). The reconstruction improvements are even more striking if the forward problem arises from parameter identification problems in PDEs with smooth solutions \cite{cj12}.

The aforementioned analysis of \eqref{eq:tik} always requires $\left\Vert \xi\right\Vert_{L^r \left(\manifold\right)} \to 0$, and is not able to fully explain the remarkable difference between $r = 1$ and $r = 2$. Recently, a new model to describe the impulsiveness of $\xi$ has been proposed by the last two authors in \cite{hw14}. Following this approach, we will assume that
\begin{equation}\label{eq:noisemodel}
\exists~\mc \in \Borel \left(\manifold\right): \qquad \left\Vert \xi\right\Vert_{L^1 \left(\mi\right)} \leq \varepsilon, \qquad \left|\mc\right| \leq \eta.
\end{equation}
Here $\Borel \left(\manifold\right)$ denotes the Borel $\sigma$-algebra of $\manifold$ and $\varepsilon, \eta \geq 0$ are noise parameters. Under this model, the main result of \cite{hw14} can be described as follows: If $F$ maps $\X$ Lipschitz continuously into a Sobolev space $W^{k,p} \left(\manifold\right)$ with $k>d/p$ and if the smoothness of the exact solution $\udag$ is described by a variational inequality with index function $\varphi$ (cf. \eqref{eq:vsc} for details; this is ensured under standard conditions as we will discuss after Assumption \ref{ass:vsc} and in Corollary \ref{cor:vsc}), then the Bregman distance between $\udag$ and $\ualdel$ is bounded by 
\begin{equation}\label{eq:rates}
\breg{\ualdel} \leq C_1 \frac{\varepsilon}{\alpha} + C_2 \frac{\eta^{2\left(\frac{k}{d} + 1\right) - \frac{2}{p}}}{\alpha^2} + C_3 (-\varphi)^*\left(-\frac{1}{\alpha}\right)
\end{equation}
with the Fenchel conjugate $(-\varphi)^*(s):=\sup_{t\geq 0}(st+\varphi(t))$ of $-\varphi$ (extended by $\infty$ on $(-\infty,0)$). 
If we set $\eta = 0$ and hence $\varepsilon = \left\Vert \xi\right\Vert_{L^1 \left(\manifold\right)}$, 
this result basically coincides with the estimates proven in the literature on regularization 
with Banach norms \cite{bo04,fh10,g10b,cj12,cjk10b}. Nevertheless, for mostly impulsive noise we expect $\varepsilon \approx 0$ and very small $\eta>0$. In this situation the remarkable improvement is given by the higher exponent $2\left(\frac{k}{d} + 1\right) - \frac{2}{p}$ of $\eta$, which is fully determined by the smoothing properties of $F$.

In this work we study \eqref{eq:tik} in case that the range of $F$ consists of analytic functions. 
This situation typically arises e.g.\  in parameter identification problems in PDEs with remote measurements. Thus we will obtain \eqref{eq:rates} for arbitrary $k$ and $p$ (probably influencing the constants), implying super-algebraic decay in $\eta$. Intuitively, as the smoothing property of the forward operator is of exponential type (if measured by the decay rate of the singular values), we may 
hope for an even better expression in $\eta$, namely an exponential one. Taking into account that in exponentially ill-posed cases the function $\varphi$ is typically only logarithmic, this then implies logarithmic rates in $\varepsilon$, but still polynomial rates in $\eta$ under an optimal choice of $\alpha$.

This study is organized as follows: In the following \S~2 we recall and adapt general 
convergence rate results for Tikhonov regularization \eqref{eq:tik} under variational 
source conditions to our noise model in combination with a new interpolation-type inequality. 
The variational source condition will be validated under classical spectral source conditions. \S~3 contains the construction and properties of Banach spaces 
$\A^\lambda\left(\manifold\right)$ of analytic functions on three manifolds 
$\manifold$ (circles, intervals, and spheres) as well as certain 
interpolation-type inequalities for these spaces. The index $\lambda$ is always 
a weight function characterizing the growth of analytic extensions. 
If the forward operator additionally maps Lipschitz continuously into $\A^\lambda\left(\manifold\right)$, then the general convergence analysis from \S~2 is applicable. This mapping property is then verified for two practical examples in \S~4, the (periodic) backwards heat equation and an inverse problem in gradiometry. We end this paper with some conclusions in \S~5.

\section{Generalized Tikhonov regularization}\label{sec:tik}

This section is devoted to the analysis of Tikhonov regularization \eqref{eq:tik}. For the whole section let $\X$ be a Banach space, $r \geq 1$ and $F:\mathcal{D}(F) \subset \X \to L^r\left(\manifold\right)$ an operator.

As discussed in \cite{hw14} a minimizer $\ualdel$ of \eqref{eq:tik} exists under reasonable assumptions on $F$ and $\R$. Furthermore, it can be proven by standard arguments that $\ualdel$ is stable w.r.t.\ $\gobs$ in a suitable sense and converges to $\udag$ as $\|\xi\|_{L^r\left(\manifold\right)} \to 0$ if $\alpha = \alpha\left(\|\xi\|_{L^r\left(\manifold\right)}\right)$ is chosen such that $\|\xi\|^r_{L^r\left(\manifold\right)}/\alpha \to 0$ and $\alpha \to 0$. For details we refer to \cite[Sec. 2.1]{hw14} and the references therein. 

\subsection{Assumptions and merit discussion}

In the following we state and discuss our assumptions used to prove rates of convergence. 
As usual an variational regularization theory will establish convergence estimates w.r.t.\ 
the Bregman distance $\breg{\ualdel}$ defined by
\begin{align*}
\breg{\ualdel}:= \R\left(f\right) - \R\left(\udag\right) - \langle f^*, f - \udag \rangle,
\end{align*}
where $f^* \in \partial \R(\udag)$ is a subgradient of $\R$ at $\udag$. Note that $\breg{f} = \left\Vert f-\udag\right\Vert_{\X}^2$ if $\R\left(f\right) = \left\Vert f-f_0\right\Vert_{\X}^2$ and $\X$ is a Hilbert space.

Moreover, as in many other recent papers we will use an abstract smoothness assumption in 
the form of a variational inequality: 
\begin{ass}[Variational source condition]\label{ass:vsc}
We assume that there exists $\beta >0$ and a concave, increasing function $\varphi : [0, \infty) \rightarrow [0, \infty)$ with $\varphi(0)=0$ such that 
\begin{equation}\label{eq:vsc}
\beta \breg{f} \leq \R(f)-\R(\udag) + \varphi \left(\|F(f)-F(\udag)\|_{L^r\left(\manifold\right)}^r\right)
\qquad\mbox{for all }f \in \mathcal{D}(F).
\end{equation}
\end{ass}

If $F : \X \to L^2\left(\manifold\right)$ is a bounded linear operator between Hilbert spaces, $\R \left(f\right) = \left\Vert f-f_0\right\Vert_{\X}^2$, $f_0 \in \X$ and $r = 2$, then it has been shown 
\cite{f12,fhm11} that \eqref{eq:vsc} is both sufficient and necessary for convergence rates of order $\varphi\left(\left\Vert \xi\right\Vert_{L^2\left(\manifold\right)}^2\right)$. Note that this furthermore shows that a spectral source condition must imply a variational source condition with a function $\varphi$ yielding the same rates of convergence in this situation. Another proof of this implication has been given by Flemming \cite{f12} for general $\varphi$. 

We are mainly interested in exponentially ill-posed problems, where the most relevant case for the function $\varphi$ in Assumption \ref{ass:vsc} is $\varphi(t)=C\varphi_p(t)$ with 
\begin{equation}\label{eq:varphip}
\varphi_p(t)= \left(-\ln\left(t\right)\right)^{-p}, \qquad t \leq \exp\left(-1\right).
\end{equation}
This function can be defined on all of $\Rset$ by concave extension, but for our asymptotic studies only the behavior close to $t = 0$ is of relevance.
As shown by Flemming \cite{f12}, if $\X$ is a Hilbert space, $r = 2$ and $F$ is bounded 
and linear, the spectral source condition 
$\udag = \varphi_p\left(\Op^*\Op\right) w$ implies a variational source condition \eqref{eq:vsc} 
with $\varphi = \beta' \varphi_{2p}$ and $\beta'>0$ depending on $\|w\|$, $\|\Op\|$, and $p$.  

For nonlinear operators $F$, the condition \eqref{eq:vsc} can be seen as a combination of source 
and nonlinearity condition (see e.g. \cite[Props. 3.35 and 3.38]{s08}). For inverse medium scattering 
problems it was recently shown with the help of complex geometrical optics solutions (cf. \cite{HW:15}) 
that a variational source condition with $\varphi$ 
of the form \eqref{eq:varphip} holds true if the solution belongs to a Sobolev ball, and $p$ 
is an explicit function of the Sobolev index. 

Both techniques for verifying Assumption \ref{ass:vsc} mentioned above yield \eqref{eq:vsc} 
with $r=2$. To obtain $L^1$-variational source conditions from $L^2$-variational source 
conditions, we will use the following smoothing properties of $F$, which also play a central role 
in the rest of our analysis: 
\begin{ass}[smoothing properties]\label{ass:interp_ineq}
Let $\X$ be a Banach space, $F: \mathcal{D}(F) \subset \X \rightarrow L^1(\manifold)$ and $\R:\X \rightarrow (-\infty, \infty]$ a convex, lower semi-continuous penalty functional. Assume that there are $\delta_0 >0$, $q>1$ and an increasing function $\gamma:[0, \delta_0] \rightarrow [0, \infty)$ with $\lim_{\delta \rightarrow 0} \gamma(\delta) =0$ such that 
\begin{align}\label{Condition3}
\|F\left(f\right) - \gdag\|_{L^\infty\left(\manifold\right)} \leq \gamma(\delta) \breg{f}^{\frac1q} + \frac{1}{\delta} \|F\left(f\right) - \gdag\|_{L^1(\manifold)}.
\end{align}
for all $0<\delta \leq \delta_0$ and all $f \in \X$.
\end{ass}

\begin{rem}\label{rem:smoothness}
Suppose there exists a Banach space $\Z\subset L^{\infty}(\manifold) \cap L^1(\manifold)$ such that
\begin{itemize}
\item[(A)] $F$ maps Lipschitz continuously into $\Z$, i.e.
\begin{align}\label{condition1}
\|F(f)-F(\udag)\|_{\Z} \leq C \breg{f}^{\frac{1}{q}}
\end{align}
with $C>0$, $q>1$ and
\item[(B)] there exists $\delta_0 >0$ and an increasing function $\gamma:[0, \delta_0) \rightarrow [0, \infty)$ with $\lim_{\delta \rightarrow 0}{\gamma(\delta)} =0$ such that for all $0< \delta \leq \delta_0$ and all $g \in \Z$
\begin{align}\label{condition2}
\|g\|_{L^\infty\left(\manifold\right)} \leq \gamma(\delta) \|g\|_{\Z} + \frac{1}{\delta} \|g\|_{L^1(\manifold)}.
\end{align}
\end{itemize}
Inserting \eqref{condition1} in \eqref{condition2} with $g =  F\left(f\right) - \gdag$ proves that Assumption \ref{ass:interp_ineq} holds true with $\gamma(\delta)$ in  \eqref{Condition3} replaced by $C \gamma(\delta)$.
\end{rem}

An inequality of the form \eqref{condition2} is shown in \cite[Lem. 3.2]{hw14} in the case 
of Sobolev spaces on a bounded Lipschitz domain $\manifold \subset \Rset^d$,  
$\Z = W^{k,p} \left(\manifold\right)$ with $k > \frac{d}{p}$. In this case we obtained $\gamma(\delta)= c_1 \delta^{\frac{k}{d} - \frac{1}{p}}$ where $c_1>0$ is a constant depending only on $\manifold, k$ and $p$, but not on $\delta \in (0, \delta_0)$. As \eqref{condition2} holds for any $0<\delta \leq \delta_0$, it can (in this special case of $\gamma$) be reformulated as
\begin{equation}\label{eq:interp_ineq_classic}
\|g\|_{L^\infty\left(\manifold\right)} \leq \inf\limits_{0<\delta \leq \delta_0} \left[c_1 \delta^{\frac{k}{d} - \frac{1}{p}}\|g\|_{W^{k,p} \left(\manifold\right)} + \frac{1}{\delta} \|g\|_{L^1(\manifold)}\right] \leq C \|g\|_{W^{k,p} \left(\manifold\right)}^{\frac1{u+1}}\|g\|_{L^1(\manifold)}^{\frac{u}{u+1}}
\end{equation}
with $u := k/d - 1/p$. If $\left\Vert g \right\Vert_{L^1\left(\manifold\right)}
\leq c_1 u \delta_0\left\Vert g \right\Vert_{W^{k,p}\left(\manifold\right)}$, 
this follows from calculating the infimum using differentiation w.r.t.\ 
$\delta$. However, since $W^{k,p}\left(\manifold\right)$ is continuously 
embedded into $L^1(\manifold)$, \eqref{eq:interp_ineq_classic} holds true 
for all $g$, with a larger constant $C$ if $c_1 u \delta_0$ is 
smaller than the embedding constant.
Consequently, \eqref{eq:interp_ineq_classic} is of the form of a classical interpolation inequality and \eqref{condition2} can be interpreted as an interpolation inequality as well.

\smallskip

The more general formulations \eqref{Condition3} and \eqref{condition2} allow for a wider class of spaces $\Z$, in particular spaces of analytic functions, which are more suitable for exponentially smoothing operators $F$. As in \eqref{eq:interp_ineq_classic} the asymptotic behavior of $\gamma \left(\delta\right)$ as $\delta \searrow 0$ will depend on the smoothing properties of $F$, more precisely the more smoothing $F$ the faster $\gamma \left(\delta\right) \searrow 0$ as $\delta \searrow 0$. In \S~3 we will introduce 
spaces $\Z$ which guarantee \eqref{condition2} with exponentially decaying $\gamma$.

\smallskip

Now we are in position to prove the following proposition which validates Assumption \ref{ass:vsc} under a spectral source condition. As we are mainly interested in exponentially ill-posed problems, we restrict ourselves to logarithmic source conditions, which are appropriate in this case.
\begin{prop}\label{prop:vsc_valid}
Let $\X$ be a Hilbert space, $\manifold \subset \Rset^d
$ and choose $\R\left(f\right) =  \left\Vert f\right\Vert_{\X}^2$, $r = 1$. Suppose that 
$\Op$ satisfies Assumption~\ref{ass:interp_ineq} with $q = 2$ and  \eqref{eq:noisemodel} holds true.  
\begin{enumerate}
\item For all $\alpha>0$, $0 \leq \eta \leq \delta_0/2$, and $\varepsilon\geq 0$ 
the estimator $\ualdel$ from \eqref{eq:tik} is bounded by
\begin{equation}\label{eq:energy_est}
\left\Vert \ualdel\right\Vert_{\X} \leq 
\left\Vert \udag\right\Vert_{\\X}+ \sqrt{\frac{2\varepsilon}{\alpha}} +\frac{2\eta\gamma\left(2\eta\right)}{\alpha}\,.
\end{equation}
\item Suppose $\udag$ satisfies the variational source condition (Assumption \ref{ass:vsc}) 
with $r=2$, $\alpha$ is chosen such that 
\begin{equation}\label{eq:PC_condition}
\sqrt{\frac{2\varepsilon}{\alpha}}+\frac{2\eta \gamma\left(2\eta\right)}{\alpha} \leq 
C_{\rm e}
\end{equation}
for all $\varepsilon, \eta$  with some constant $C_{\rm e}$, and $\Op$ satisfies the 
Lipschitz condition \eqref{condition1}. Then there exists $C_{\rm v}>0$ independent of 
$\alpha,\varepsilon$ and $\eta$ such that $\udag$ satisfies the variational source 
condition \eqref{eq:vsc} with $r=1$ 
for all $\sol= \ualdel$ and $\varphi$ replaced by $\varphi(C_{\rm v}\cdot)$. 
\end{enumerate}
\end{prop}
\begin{proof}
\emph{Part 1:} Using the minimizing property 
\begin{align}
\frac{1}{\alpha} \left\Vert \Op(\ualdel) -\gobs\right\Vert_{L^1(\manifold)} + \R(\ualdel) \leq \frac{1}{\alpha} \|\xi\|_{L^1(\manifold)} + \R(\udag) \label{Inequality11}
\end{align}
and the noise model \eqref{eq:noisemodel} we find the estimate
\begin{equation} \label{eq:aux3}
\left\Vert \Op(\ualdel) - \gobs\right\Vert_{L^1\left(\mi\right)}  + \alpha \left\Vert 
\ualdel\right\Vert_{\X}^2 \leq \varepsilon+ \alpha \left\Vert \udag\right\Vert_{\X}^2 + \left\Vert \Op
(\ualdel) - \gdag\right\Vert_{L^1 \left(\mc\right)} 
\end{equation}
Using \eqref{Condition3} with $\delta = 2 \eta$ and $\breg{\ualdel}^{1/q}= \|\ualdel-\udag\|\leq 
\|\ualdel\|+\|\udag\|$ 
(since $\X$ is a Hilbert space) it follows
\[
\left\Vert \Op(\ualdel) - \gdag\right\Vert_{L^1 \left(\mc\right)}  
\leq \frac12 \left\Vert \Op(\ualdel) - \gdag \right\Vert_{L^1 \left(\manifold\right)} 
+ \eta \gamma\left(2\eta\right) \left(\left\Vert \ualdel \right\Vert_{\X} 
+ \left\Vert \udag\right\Vert_\X\right)
\]
which then implies
\begin{equation}\label{eq:aux2}
\left\Vert \Op(\ualdel) - \gdag\right\Vert_{L^1 \left(\mc\right)} \leq \left\Vert \Op(\ualdel) - \gdag \right\Vert_{L^1 \left(\mi\right)} + 2\eta \gamma\left(2\eta\right) \left(\left\Vert \ualdel \right\Vert_{\X} + \left\Vert \udag\right\Vert_\X\right).
\end{equation}
As
\(
\left|\left\Vert \Op(\ualdel) - \gobs\right\Vert_{L^1\left(\mi\right)} - \left\Vert \Op(\ualdel) - \gdag \right\Vert_{L^1 \left(\mi\right)}  \right| \leq \varepsilon,
\)
we find from inserting \eqref{eq:aux2} into \eqref{eq:aux3} that
\[
\left\Vert \ualdel\right\Vert_{\X}^2 \leq \frac{2\varepsilon}{\alpha} +\left\Vert \udag\right\Vert_{\X}^2 + \frac{2\eta \gamma\left(2\eta\right)}{\alpha} \left(\left\Vert \ualdel \right\Vert_{\X} + \left\Vert \udag\right\Vert_\X\right).
\]
Now we add 
$-\frac{2\eta \gamma\left(2\eta\right)}{\alpha} \left\Vert \ualdel \right\Vert_{\X}+
\frac{\eta^2 \gamma\left(2\eta\right)^2}{\alpha^2}$ on both sides, complete the squares and obtain
\[
\left(\left\Vert \ualdel \right\Vert - \frac{\eta \gamma \left(2 \eta\right)}{\alpha} \right)^2 \leq \frac{2 \varepsilon}{\alpha} + \left( \left\Vert \udag\right\Vert + \frac{\eta \gamma \left(2 \eta\right)}{\alpha} \right)^2 \leq  \left(\sqrt{\frac{2 \varepsilon}{\alpha}} + \left\Vert \udag\right\Vert + \frac{\eta \gamma \left(2 \eta\right)}{\alpha} \right)^2.
\]
Taking square roots yields the desired estimate.\\
\emph{Part 2:} Note that 
$
\left\Vert \Op\left(f\right) - \gdag\right\Vert_{L^2\left(\manifold\right)}^2 \leq \left\Vert \Op\left(f\right) - \gdag\right\Vert_{L^\infty\left(\manifold\right)}\left\Vert \Op\left(f\right) - \gdag\right\Vert_{L^1\left(\manifold\right)}.
$
Furthermore \eqref{Condition3} together with \eqref{condition1} yields
\begin{align*}
\left\Vert \Op\left(f\right) - \gdag\right\Vert_{L^\infty\left(\manifold\right)} &\leq \gamma\left(\delta\right) \left\Vert f-\udag\right\Vert_{\X} + \frac{1}{\delta} \left\Vert \Op\left(f\right) - \gdag\right\Vert_{L^1\left(\manifold\right)} \\
&\leq \left(\gamma\left(\delta\right)+ \frac{C}{\delta} \right) \left\Vert f- \udag\right\Vert_{\X} 
\end{align*}
for arbitrary $\delta$. Together with the first part this shows that 
$\| \Op\left(f\right) - \gdag\|_{L^2(\manifold)}^2\leq 
C_{\rm v}\|\Op\left(f\right) - \gdag\|_{L^1(\manifold)}$ with 
$C_{\rm v}:= (\gamma\left(\delta\right)+C/\delta)(2\|\udag\|_{\X}+C_{\rm e})$, which 
yields the assertion. 
\end{proof}

As mentioned after \eqref{eq:varphip}, spectral logarithmic source conditions imply 
variational logarithmic source conditions. As $\varphi_{2p}(C_{\rm v}x)=\varphi_{2p}(x)+o(1)$ 
as $x\searrow 0$, we obtain: 
\begin{cor}\label{cor:vsc}
If $\X$ is a Hilbert space, $F:\X \to L^2(\manifold)$ is linear, the spectral source 
condition $\udag=\varphi_p(F^*F)w$ is satisfied, and \eqref{eq:PC_condition} holds true. 
Then the variational source condition \eqref{eq:vsc} holds true for $\sol = \ualdel$ 
with $r = 1$, and $\varphi = \beta'\varphi_{2p}$ where 
$\beta'>0$ depends only on $\left\Vert w \right\Vert_{\X}$, $\|F\|$, $p$, $\delta_0$, and $C_{\rm e}$. 
\end{cor}

\subsection{Rates of convergence}

Now we will present error estimates for the minimizers $\ualdel$ in \eqref{eq:tik}. 
Our analysis uses the observation by Grasmair \cite{g10b} that the approximation error can be bounded by the Fenchel conjugate function of $-\varphi$. Let us introduce a function $\psi$ by
\begin{equation}\label{eq:defi_associate}
\psi( \alpha):= (-\varphi)^* \left( -\frac{1}{\alpha}\right) = \sup\limits_{\tau\geq 0} \left(  \varphi(\tau)-\frac{\tau}{\alpha} \right), \qquad \alpha>0.
\end{equation}
Note from this  definition that $\psi$ is concave and non-decreasing.

For $\varphi_p$ defined in \eqref{eq:varphip} one obtains $\psi_p(\alpha)= \left( \ln \frac{1}{\alpha p}\right)^{-p} \left( 1+o(1)\right)$ as $\alpha \rightarrow 0$. 

\medskip
We are now in a position to formulate and prove the main result of this section, which 
closely follows \cite[Thm. 3.4]{hw14}:
\begin{thm}\label{Theorem1}
Let $\X$ be a Banach space, $F:\mathcal{D}(F) \subset \X \rightarrow L^1 \left(\manifold\right)$ be an operator and $\R: \X \rightarrow(-\infty, \infty]$ a convex, lower semi-continuous functional. 
Let $\ualdel$ be a solution of \eqref{eq:tik} with $r = 1$ and suppose that the noise function $\xi$ fulfills \eqref{eq:noisemodel}. 
If Assumptions \ref{ass:vsc} and \ref{ass:interp_ineq} hold true with $q>1$ (actually \eqref{eq:vsc} 
is only needed for $\sol=\ualdel$), then for any $b>1$ there exists $\eta_0>0$ such that
\begin{align}
\beta \breg{\ualdel} &\leq \frac{2q'}{\alpha} \epsilon + \beta^{1-q'}\left( \frac{2}{\alpha} \eta \gamma \left( \frac{2b}{b-1}\eta \right)\right)^{q'} + q' \psi(b\alpha),\label{Inequality4}\\
\frac{1}{b} \|F(\ualdel)-\gdag\|_{L^1(\manifold)} &\leq 4q'\epsilon + 2\frac{\left( 2\eta \gamma\left( \frac{2b}{b-1} \eta\right)\right)^{q'}}{(\beta \alpha)^{q'-1}} + 2\alpha q'\psi(2b\alpha)\label{Inequality5}
\end{align}
for all $0<\eta\leq \eta_0,~ \epsilon \geq 0$ and $\alpha>0$. Here and below $q'$ denotes the conjugate index of $q$, i.e.\ $\frac{1}{q} + \frac{1}{q'} =1$. 
\end{thm}

\begin{proof}
By assumption there is a measurable $\mc$ such that $\|\xi\|_{L^1(\mi)} \leq \epsilon$ and $\left|\mc\right| \leq \eta$. Thus we have for any $g \in L^1(\manifold)$ that
\begin{align}
&\|g-\gdag\|_{L^1(\manifold)} + \|\xi\|_{L^1(\manifold)} 
= \|g-\gdag\|_{L^1(\mi)} + \|g-\gdag \|_{L^1(\mc)} + \|\xi \|_{L^1(\manifold)}\nonumber\\
\leq& \left( \|g-\gobs\|_{L^1(\mi)} + \|\xi\|_{L^1(\mi)} \right) + \|g-\gdag\|_{L^1(\mc)} + \|\xi\|_{L^1(\manifold)}\nonumber\\
\leq& \|g-\gobs\|_{L^1(\manifold)} + 2 \|\xi\|_{L^1(\manifold\setminus \mc)} +2 \| g-\gdag\|_{L^1(\mc)}\nonumber\\
\leq& \| g-\gobs \|_{L^1(\manifold)} + 2\epsilon + 2 |\mc| \| g-\gdag\|_{L^\infty\left(\manifold\right)}.\label{Inequality8}
\end{align}
For $g = F(\ualdel)$ we will now apply the interpolation inequality \eqref{condition2} with $\delta:= \frac{2b}{b-1}\eta$. Given $b>1$ we set $\eta_0:= \frac{b-1}{2b} \delta_0$ and obtain for any $0<\eta \leq \eta_0$ that
\begin{multline}\label{Inequality15}
\frac{1}{b}  \|F(\ualdel)-\gdag\|_{L^1(\manifold)} + \|\xi\|_{L^1(\manifold)} \\\leq \|F(\ualdel)-\gobs\|_{L^1(\manifold)} + 2 \left(\epsilon + \eta \gamma\left(\frac{2b}{b-1}\eta\right)\breg{\ualdel}^{\frac1q}\right).
\end{multline}
Now Assumption \ref{ass:vsc} implies
\begin{align}
&\beta \breg{\ualdel} \nonumber\\
\stackrel{\eqref{eq:vsc}}{\quad\leq\quad}& \R(\ualdel) - \R(\udag) + \varphi \left(\|F(\ualdel) - F(\udag)\|_{L^1(\manifold)}\right)\nonumber\\
\stackrel{\eqref{Inequality11}}{\quad\leq\quad}& \frac{1}{\alpha} \|\xi\|_{L^1(\manifold)} - \frac{1}{\alpha}\|F(\ualdel)-\gobs\|_{L^1(\manifold)}+ \varphi \left(\|F(\ualdel) - F(\udag)\|_{L^1(\manifold)}\right) \nonumber\\
\stackrel{\eqref{Inequality15}}{\quad\leq\quad}& \frac{2}{\alpha} \left( \epsilon + \eta \gamma\left(\frac{2b}{b-1}\eta\right)\breg{\ualdel}^{\frac1q}\right) - \frac{1}{b\alpha} \|F(\ualdel)-\gdag\|_{L^1(\manifold)}\nonumber\\
&\qquad\qquad\qquad\qquad + \varphi \left(\|F(\ualdel) - F(\udag)\|_{L^1(\manifold)}\right)\label{estimates} \\
\quad\leq\quad&
\frac{2}{\alpha} \epsilon + \frac{2\eta}{\alpha} \gamma\left(\frac{2b}{b-1}\eta\right)\breg{\ualdel}^{\frac{1}{q}}+ \psi(b\alpha) \nonumber.
\end{align}
The second term on the right-hand side can be handled by Young's inequality $xy \leq x^{q'}/q' + y^q/q$ with $x= \beta^{-1/q}\frac{2\eta}{\alpha} \gamma\left(\frac{2b}{b-1}\eta\right)$, which implies that \eqref{Inequality4} holds true.

Since $\breg{\ualdel}\geq 0$, we derive from \eqref{estimates} and \eqref{Inequality4} in a similar fashion that
\begin{align*}
& \frac{1}{2b\alpha}\|F(\ualdel)-\gdag\|_{L^1(\manifold)}\\
\stackrel{\eqref{estimates}}{\quad\leq\quad}& \frac{2}{\alpha} \left( \epsilon + \eta \gamma\left(\frac{2b}{b-1}\eta\right)\breg{\ualdel}^{\frac1q}\right) - \frac{1}{2b\alpha} \|F(\ualdel)-\gdag\|_{L^1(\manifold)}\\
&\qquad\qquad\qquad\qquad + \varphi \left(\|F(\ualdel) - F(\udag)\|_{L^1(\manifold)}\right) \\
\quad\leq\quad&\frac{2}{\alpha}\epsilon + \frac{2}{\alpha}\eta \gamma\left(\frac{2b}{b-1}\eta\right)\breg{\ualdel}^{\frac{1}{q}} + \psi(2b\alpha)\\
\quad\leq\quad& \frac{2}{\alpha}\epsilon + \frac{1}{q'\beta^{q'-1}} \left( \frac{2}{\alpha} \eta \gamma\left(\frac{2b}{b-1}\eta\right)\right)^{q'} + \frac{1}{q} \beta \breg{\ualdel} + \psi(2b\alpha)\\
\stackrel{\eqref{Inequality4}}{\quad\leq\quad}& 
\frac{2q'}{\alpha}\epsilon + \beta^{1-q'}\left( \frac{2}{\alpha} \eta \gamma\left(\frac{2b}{b-1}\eta\right)\right)^{q'} + q' \psi(2b\alpha),
\end{align*}
where we used that $\psi$ is monotonically increasing. This immediately proves \eqref{Inequality5}.
\end{proof}

\begin{rem}
The optimal $\alpha$ can be obtained by balancing out the terms, cf. \cite[Thm. 2.3, 3.]{hw14}. For simplicity we will do this only in the specific examples in \S~\ref{sec:ex}.\\
Note that the phenomenon of \textit{exact penalization} (cf. \cite{bo04}) will occur if $\varphi \left(x\right) = c\cdot x$ with some $c>0$, as then $\psi\left(\alpha\right) = 0$ if $\alpha \leq 1/c$ and $\psi\left(\alpha\right) = \infty$ otherwise. Consequently, in this situation the optimal parameter is $\alpha = 1/c$, and for noise-free data this already yields an exact reconstruction.
\end{rem}

\begin{rem}\label{rem:pc}
All choices of $\alpha$ for which the right-hand side of \eqref{Inequality4} 
tends to $0$ (e.g.\ a-priori choices which approximately minimize this 
right-hand side) satisfy the condition \eqref{eq:PC_condition}.
\end{rem}

\section{Spaces of analytic functions and $L^\infty$-$L^1$ interpolation inequalities}
Let $\A \left(\manifold\right)$ denote the set of all analytic functions on $\manifold$. 
In this section we will introduce spaces $\A^{\lambda}(\manifold) \subsetneq \A \left(\manifold\right) $ of 
analytic functions on the manifolds $\manifold=\Rset/2\pi\Zset$, $(-1,1)$, and $\SS^2$  
and establish interpolation inequalities for spaces 
$L^\infty(\manifold)\subset L^1(\manifold)\subset \A^{\lambda}(\manifold)$. 

\subsubsection*{General strategy}
We will introduce classes of analytic functions $\A^{\lambda}(\manifold)$ which are indexed 
by an increasing function $\lambda:[0,\infty)\to  \Rset\cup\{\infty\}$. $\lambda$ will be 
a bound on the growth rate of analytic extensions of functions $g\in \A^{\lambda}(\manifold)$. Moreover, for 
each choice of $\manifold$ we will introduce a sequence of finite dimensional 
subspaces $\calU_m\subset \A^{\lambda}(\manifold)$, $m\in\Nset$ 
with corresponding orthogonal projections 
$P_m:L^2(\manifold)\to \calU_m$ and prove exponential bounds of the approximation error 
of the form 
\begin{equation}\label{eq:approx_bound}
\|g-P_mg\|_{L^{\infty}(\manifold)}\leq q_{\lambda}(m)
\exp(-\lambda^*(m))\|g\|_{\A^{\lambda}(\manifold)}
\end{equation}
for $m\geq m_0$ and $g\in \A^{\lambda}(\manifold)$ 
with the Fenchel-conjugate $\lambda^*(s):=\sup_{x\geq 0}[rs-\lambda(r)]$ of the function 
$\lambda$ (extended by $\infty$ on $(-\infty,0)$) and a polynomial $q_{\lambda}$. 
As a second ingredient we will need that the projection operators $P_m$ 
can be extended to $P_m: L^1(\manifold) \to L^{\infty}(\manifold)$ and satisfy 
the operator norm bound
\begin{equation}\label{eq:Nikolskii}
\|P_mg\|_{L^\infty(\manifold)}\leq \kappa_m \|g\|_{L^1(\manifold)}\qquad \mbox{for all }
g\in L^1(\manifold). 
\end{equation}
This is the case if $P_m$ has an integral representation  
$(P_mg)(x)=\int_{\manifold}K_m(x,y)g(y)\,\mathrm{d}y$ and 
$\|K_m\|_{L^{\infty}(\manifold\times\manifold)}=\kappa_m$. The kernel $K_m$ is given 
by $K_m(x,y)= \sum_{j=1}^{N_m} g_j(x)g_j(y)$ for any orthonormal basis 
$\{g_1,\dots,g_{N_m}\}\subset \calU_m$. 
With the help of \eqref{eq:approx_bound} and \eqref{eq:Nikolskii} we obtain
\begin{align*}
\|g\|_{L^{\infty}(\manifold)}
&\leq \|g-P_mg\|_{L^{\infty}(\manifold)}+ \|P_mg\|_{L^{\infty}(\manifold)}\\
&\leq q_{\lambda}(m)\exp(-\lambda^*(m))\|g\|_{\A^{\lambda}(\manifold)} 
+ \kappa_m\|g\|_{L^1(\manifold)}
\end{align*}
for $m\geq m_0$ and $g\in \A^{\lambda}(\manifold)$. 
We then choose $m(\delta)$ (decreasing in $\delta>0$) 
such that $\kappa_{m(\delta)}\leq \delta^{-1}$. This yields the interpolation inequality
\begin{align}\label{eq:general_interp_ieq}
\begin{aligned}
\|g\|_{L^{\infty}(\manifold)} 
&\leq \gamma(\delta)\|g\|_{\A^{\lambda}(\manifold)} 
+ \frac{1}{\delta}\|g\|_{L^1(\manifold)}\quad\mbox{with}\\
\gamma(\delta)&:=  q_{\lambda}(m(\delta))\exp(-\lambda^*(m(\delta)))
\end{aligned}
\end{align}
for all $g\in \A^{\lambda}(\manifold)$ and $\delta\leq\delta_0$ if $m(\delta_0)=m_0$ 
(see \eqref{condition2}).

\subsection{Case $\manifold=\Rset/2\pi\Zset$} 
Note that any $2\pi$-periodic real-analytic function $g: \Rset/ 2\pi\Zset \to \Rset$ can be 
extended to a holomorphic function $\tilde g$ on some strip $S_B = \left\{z \in \Cset ~\big|~ \left|\Im\left(z\right)\right| < B\right\}$ with $B>0$. We may measure the smoothness of $g$ in terms of $B$ and the growth of $\left|\tilde g\left(x+\textup{i}y\right)\right|$ as $\left|y\right| \nearrow B$. This leads to the following definition (cf.\ \cite[Sec. 11]{kress1999linear} for the special case 
that $\lambda$ is an indicator function):
\begin{Def}
Let $\lambda : \left[0,\infty\right) \to \Rset \cup \left\{\infty\right\}$ be an increasing weight function with positive $B_\lambda := \sup\left\{ r \in \left[0,\infty\right) ~\big|~\lambda\left(r\right) < \infty\right\}$. With the holomorphic 
extension $\tilde{g}$ of $g$ on $B_\lambda$ introduced above 
we define the space $\A^\lambda\left(\Rset / 2 \pi \Zset\right)$ by
\begin{align*}
\A^\lambda\left(\Rset / 2 \pi \Zset\right) &:= \left\{g \in C\left(\Rset / 2 \pi \Zset\right) ~\big|~ \tilde g\text{ exists on }S_{B_\lambda},\quad \left\Vert g\right\Vert_{\A^\lambda\left(\Rset / 2 \pi \Zset\right)} < \infty\right\},\\
\left\Vert g\right\Vert_{\A^{\lambda}\left(\Rset / 2 \pi \Zset\right)}&:= \sup\limits_{z \in S_{B_\lambda}}\left[\exp\left(-\lambda\left(\left|\Im\left(z\right)\right|\right)\right)\left|\tilde g \left(z\right)\right|\right].
\end{align*}
\end{Def}
\begin{thm}\label{thm:periodic}
Let $\lambda : \left[0,\infty\right) \to \Rset \cup \left\{\infty\right\}$ be non-decreasing 
with $B_{\lambda}>0$. 
\begin{enumerate}
\item The space $\A^\lambda\left(\Rset / 2 \pi \Zset\right)$ equipped with $\left\Vert \cdot \right\Vert_{\A^{\lambda}\left(\Rset / 2 \pi \Zset\right)}$ is a Banach space.
\item The Fourier coefficients $\hat g \left(n\right)$ of $g \in \A^\lambda\left(\Rset / 2 \pi \Zset\right)$ satisfy
\begin{equation}\label{eq:periodic_fourier_est}
\left| \hat g \left(n\right) \right| \leq \exp\left(-\lambda^* \left(\left|n\right|\right) \right) \left\Vert g \right\Vert_{\A^{\lambda}\left(\Rset / 2 \pi \Zset\right)}, \qquad n \in \Zset.
\end{equation}
\item If $P_m$ denotes the $L^2$-orthogonal projection onto $\mathrm{span} \left\{ \exp\left( \textup{i} n \cdot\right)~\big|~ \left|n\right| \leq m\right\}$, then there exists $m_0 \in \Nset$ and a constant $c_{\lambda}>0$ such that
\begin{equation}\label{eq:periodic_approx_1}
\left\Vert \left(I-P_m\right)g\right\Vert_{L^\infty\left(-\pi,\pi\right)}\leq 2 c_{\lambda} \exp\left(-\lambda^*\left(m\right) \right) \left\Vert g \right\Vert_{\A^{\lambda}\left(\Rset / 2 \pi \Zset\right)}
\end{equation}
for all $g \in \A^\lambda\left(\Rset / 2 \pi \Zset\right)$ and $m \geq m_0$.
\item There exists $\delta_0>0$ such that the interpolation inequality \eqref{eq:general_interp_ieq} 
holds true with $\manifold =\Rset/2\pi\Zset$ and  
\[
\gamma(\delta) =  c_{\lambda}
\exp \left( -\lambda^* \left( \left\lfloor \frac{\pi}{\delta} - \frac{1}{2}\right\rfloor\right)\right)
\]
for all $g\in\A^{\lambda}(\Rset/2\pi\Zset)$ 
and all $0<\delta\leq \delta_0$. Here $\lfloor x \rfloor:=\sup\{n\in\Zset:n\leq x\}$ for $x\in\Rset$.
\end{enumerate}
\end{thm}
\begin{proof}
\emph{Part 1:} Obviously, $\A^{\lambda}\left(\Rset / 2 \pi \Zset\right)$ 
is a normed space. Since any Cauchy 
sequence in this spaces converges uniformly on compact subsets of $B_\lambda$, 
the pointwise limit is again holomorphic, so it belongs to 
$\A^\lambda\left(\Rset / 2 \pi \Zset\right)$. This follows from the Cauchy integral formula and proves completeness. \\
\emph{Part 2:} Let $g \in \A^\lambda\left(\Rset / 2 \pi \Zset\right)$ and denote by $\tilde{g}$ the holomorphic extension of $g$ to $S_{B_\lambda}$. Since $\tilde{g}$ and $\tilde{g}(2 \pi + \cdot)$ coincide on $\Rset$ by the periodicity of $g$, it follows from the identity principle that $\tilde{g}$ and $\tilde{g}(2 \pi + \cdot)$ also coincide on $S_{B_{\lambda}}$, i.e.\  $\tilde{g}$ is $2 \pi-$periodic.

For $0 < R < B_\lambda$ consider the annulus $K_R:= \{ w \in \Cset ~\big \vert~ \exp(-R) < |w| < \exp(R)\}$. The map $\theta: S_R \rightarrow K_R, \theta(z):= \exp\left(\textup{i}z\right)$ is holomorphic for any $0 < R \leq B_\lambda$. The restriction of $\theta$ to the rectangle $D_R:= \{z \in S_R ~\big|~ \Re\left(z\right) \in (-\pi, \pi]\}$ is bijective and $\theta\left(-\pi +\textup{i}y\right) = \theta\left(\pi +\textup{i}y\right) = -\exp(-y).$ Hence $h:= \tilde{g} \circ \theta^{-1}: K_{B_\lambda} \rightarrow \Cset$ is continuous on the line segment $\{ -\exp(-y)~\big|~ -R <y<R \}$ and therefore holomorphic in the annulus $K_{B_\lambda}.$ If follows that for all $w = \exp\left(\textup{i}x\right) \in \SS^1, x \in (-\pi, \pi]$ we have 
\begin{equation}\label{eq:aux5}
h(w) = g \circ \theta^{-1}(w) = g(x)= \sum\limits_{n \in \Zset} \hat{g}(n) \exp\left(\textup{i} n x\right)= \sum\limits_{n \in \Zset} \hat{g}(n) w^n.
\end{equation}
Being holomorphic in the annulus $K_{B_\lambda}$, $h$ has a Laurent series expansion which by the identity principle is uniquely determined by the series expansion \eqref{eq:aux5} on $\SS^1$, i.e.\ 
$h(w) = \sum_{n \in \Zset} \hat{g}(n) w^n$ for all $w \in K_{B_\lambda}$.
For any $s \in \Rset$ with $\exp\left(-B_\lambda\right) < s< \exp\left(B_\lambda \right)$ we have
\[
|\hat{g}(n)|  = \bigg|\frac{1}{2\pi i} \int_{|w|=s} h(w) \frac{\,\mathrm dw}{w^{n+1}}\bigg| \leq 
\max\limits_{|w|=s} \left|h(w)\right| \cdot \frac{1}{s^n}.
\]
But $\max_{\left|w\right| = s}\left|h(w)\right|=  \max_{\Im\left(z\right) =  \ln\left(1/s\right)}\left|\tilde g\left(z\right)\right|$, and hence with $r:= \ln\left(1/s\right)$ this implies
\begin{align*}
\left|\hat{g}(n)\right| 
&\leq \max\limits_{\Im\left(z\right) = \pm r} \left|\tilde{g}(z)\right| \cdot \exp(-r\left|n\right|)
\leq \left\Vert g\right\Vert_{\A^{\lambda}(\Rset / 2 \pi \Zset)} \cdot \exp \left( -\left[r\left|n\right| - \lambda(r)\right] \right)
\end{align*}
for all $0 < r < B_\lambda$. Optimizing in $r$ proves \eqref{eq:periodic_fourier_est}.

\emph{Part 3:} The orthogonal projection is given by $P_m g(x)= \sum_{|n| \leq m} \hat{g}(n) \exp\left(\textup{i}nx\right)$, $x \in (-\pi, \pi)$. Using \eqref{eq:periodic_fourier_est} we obtain 
\begin{align*}
\|g-P_m g\|_{L^{\!\infty}\!\left(-\pi,\pi\right)} 
&= \!\!\!\!\!\sup\limits_{x \in (-\pi, \pi)} \bigg|\!\!\sum\limits_{|n|>m} \!\!\hat{g}(n) \exp\left(\textup{i}nx\right)\!\bigg| 
\leq 2 \left\Vert g\right\Vert_{\A^{\lambda}(\Rset / 2 \pi \Zset)} \!\!\!\sum^{\infty}_{n=m+1} 
\!\!\!\!\exp\left(-\lambda^*(n)\right)).
\end{align*}
Recall that $\lambda^*$ is always convex. As $\lambda(r)=\infty$ for $r<0$, $\lambda^*(s)=\sup_{r\geq 0}
[sr-\lambda(r)]$ is non-decreasing, and since $B_\lambda>0$ it is easy to see that 
there exists $m_0 \in \Nset$ with $\lambda^*(m_0)>0$ and $a:= \lambda^*(m_0+1)- \lambda^*(m_0) >0$. 
Thus the convexity of $\lambda^*$ implies
$a \leq \lambda^*(m+1)-\lambda^*(m)$ for all $m \geq m_0$. Consequently, for any $n > m \geq m_0$, we have
$\lambda^*(n) -\lambda^*(m) = \sum_{j=m}^{n-1} \left( \lambda^*(j+1)-\lambda^*(j)\right) \geq (n-m)a$.
Therefore, 
\begin{align*}
\sum_{n=m+1}^{\infty} \exp\left(- \left[\lambda^*(n) - \lambda^*(m)\right]\right)
\leq \sum_{n=m+1}^{\infty} \exp(-a(n-m))= 
\frac{1}{\exp(a)-1} 
\end{align*}
whenever $m \geq m_0$.  This shows \eqref{eq:periodic_approx_1} with 
$c_{\lambda}:=2/(\exp(a)-1)$. 

\emph{Part 4:} We follow our general strategy with 
$K_m(x,y):= \frac{1}{2\pi}\frac{\sin((m+1/2)(x-y)}{\sin((x-y)/2)}$ 
(Dirichlet kernel), $\kappa_m:=\frac{2m+1}{2\pi}$,
and $m(\delta)= \left\lfloor \frac{\pi}{\delta} - \frac{1}{2}\right\rfloor$. 
\end{proof}

\subsection{Case $\manifold = (-1,1)$}
The main idea is to extend a real-analytic function $g$ on $\left(-1,1\right)$ to a holomorphic function $\tilde g$ on an ellipse with foci $\{ -1,1\}$ of the form
\begin{align*}
\E_r= \left\{ x+\textup{i}y \in \Cset ~\Bigg|~ x,y \in \Rset, \frac{x^2}{\cosh\left(r\right)^2} 
+ \frac{y^2}{\sinh\left(r\right)^2} \leq 1 \right\}, \qquad r >0.
\end{align*} 

\begin{Def}
Let $\lambda : \left[0,\infty\right) \to \Rset \cup \left\{\infty\right\}$ be an increasing weight function with positive $B_\lambda := \sup\left\{ r \in \left[0,\infty\right) ~\big|~\lambda\left(r\right) < \infty\right\}$. We define the space $\A^\lambda\left(-1,1\right)$ by
\begin{align*}
\A^\lambda\left(-1,1\right)&:= \left\{g \in \A\left(-1,1\right) 
~\big|~ \tilde g\text{ exists on }\E_{B_\lambda},\quad \left\Vert g\right\Vert_{\A^\lambda\left(-1,1\right)} < \infty\right\},\\
\left\Vert g\right\Vert_{\A^{\lambda}\left(-1,1\right)}&
:= \sup\limits_{0 < t < B_\lambda} \left[\exp\left(-\lambda\left(t\right)\right) 
\sup_{z \in \partial \E_{t}}\left|\tilde g \left(z\right)\right|\right].
\end{align*}
\end{Def}

\begin{thm}\label{thm:interval}
Let $\lambda : \left[0,\infty\right) \to \Rset \cup \left\{\infty\right\}$ be non-decreasing 
with $B_{\lambda}>0$. 
\begin{enumerate}
\item The space $\A^\lambda\left(-1,1\right)$ equipped with the norm 
$\left\Vert \cdot \right\Vert_{\A^{\lambda}\left(-1,1\right)}$ is a Banach space.
\item The coefficients $a_n(g)$ of $g \in \A^\lambda\left(-1,1\right)$ with respect to 
the Chebychev polynomials satisfy
\begin{equation}\label{eq:non-periodic_fourier_est}
\left| a_n(g) \right| \leq 2\exp\left(-\lambda^*\left(\left|n\right|\right) \right) \left\Vert g\right\Vert_{\A^{\lambda}\left(-1,1\right)}, \qquad n \in \Nset.
\end{equation}
\item If $P_m$ denotes the $L^2$-orthogonal projection onto the space of polynomials of 
degree $\leq m$, then there exists $m_0 \in \Nset$ and 
a constant $c_{\lambda}>0$ such that
\begin{equation}\label{eq:non-periodic_approx_1}
\left\Vert \left(I-P_m\right)g\right\Vert_{L^\infty\left(-1,1\right)} 
\leq c_{\lambda} \exp\left(-\lambda^*\left(m\right) \right) \left\Vert g\right\Vert_{\A^{\lambda}\left(-1,1\right)}
\end{equation}
for all $g \in \A^\lambda\left(-1,1\right)$ and $m \geq m_0$.
\item There exists $\delta_0>0$ such that the interpolation inequality \eqref{eq:general_interp_ieq} 
holds true with $\manifold =(-1,1)$ and  
\begin{equation}\label{eq:interp_ineq_interval}
\gamma(\delta) = c_{\lambda}
\exp \left(-\lambda^*\left( \left\lfloor \sqrt{\frac{2}{\delta}} \right\rfloor-1\right)\right)
\end{equation}
for all $g\in\A^{\lambda}(-1,1)$ and all $\delta\leq \delta_0$. 
\end{enumerate}
\end{thm}
\begin{proof}
\emph{Part 1:} This can be shown as in the proof of Theorem \ref{thm:periodic}. \\
\emph{Part 2:} The following approximation argument is taken from Kress \cite[Thm. 11.7]{kress1999linear}.  The function 
$\theta:K_{B_\lambda} \to \E_{B_\lambda}$, 
$\theta(w):= \frac{1}{2} \left( w+\frac{1}{w}\right)$ on an annulus 
$K_{B_\lambda}$ as defined in the proof of Theorem \ref{thm:periodic}
is surjective and holomorphic (but not injective). 

Now let $g \in \A^{\lambda}\left(-1,1\right)$ with holomorphic extension 
$\tilde{g} : \E_{B_\lambda} \rightarrow \Cset$. Define 
$h:K_{B_\lambda} \rightarrow \Cset$ by $h(w):= 2 \tilde{g} \circ \theta(w).$ 
The holomorphic function $h$ in the annulus $K_{B_\lambda}$ can be expanded into a Laurent series $h(w)= \sum_{n \in \Zset} a_n w^n, w \in K_{B_\lambda}$ with coefficients 
\begin{align*}
a_n = \frac{1}{\pi i} \int_{|w|=\exp(r)} \tilde{g} \left( \frac{1}{2} \left( w+ \frac{1}{w}\right)\right)\frac{\,\mathrm dw}{w^{n+1}},
\qquad -B_\lambda <r<B_\lambda,\; n \in \Zset.
\end{align*}
Substitution $\tilde{w}=\frac{1}{w}$ shows that $a_{-n}=a_n$. Therefore
\begin{align*}
h(w)=a_0 + \sum_{n\in \Nset} a_n \left( w^n + \frac{1}{w^n}\right), \qquad w \in K_{B_\lambda}.
\end{align*}
For $|w|=1$ we write $w= \exp\left(\textup{i}t\right)$ and obtain
\begin{align*}
\frac{1}{2}\left( w^n + \frac{1}{w^n}\right) = \cos(nt) = T_n(\cos(t)) = T_n \left( \frac{1}{2}\left( w+\frac{1}{w}\right)\right),
\end{align*}
where $T_n$ is the $n-$th Chebychev polynomial. For $z \in \E_{B_\lambda}$ we find $w \in K_{B_\lambda}$ with $z = \theta\left(w\right)$ to obtain
\begin{align*}
\tilde{g}(z)= \tilde{g} \circ \theta(w) = \frac{1}{2}h(w) = \frac{a_0}{2} + \sum\limits_{n \in \Nset} a_n T_n(z).
\end{align*}
The formula for $a_n$ yields for any $n \in \Nset$ and $0\leq r<B_\lambda$ that
\begin{equation}\label{eq:aux7}
|a_n| \leq \frac{2}{\exp(rn)} \max\limits_{z \in \partial \E_{r}} |\tilde{g}(z)| 
\leq 2 \exp\left(\lambda(r)-nr\right) \left\Vert g\right\Vert_{\A^{\lambda}\left(-1,1\right)}
\end{equation}
where we used the definition of $\left\Vert \cdot\right\Vert_{\A^{\lambda}\left(-1,1\right)}$. Optimizing in $r$ yields \eqref{eq:non-periodic_fourier_est}.

\emph{Part 3:} 
As  $P_m g \left(x\right) \!\!= \!\!a_0(g)/2 + 
\sum_{n=1}^{m} a_n(g) T_n(x)$ and $\|T_n\|_{L^{\infty}(-1,1)}=1$, we obtain
\begin{align}\label{eq:aux}
\begin{aligned}
\|g-P_m g\|_{L^\infty\left(-1,1\right)}  &= \sum_{n=m+1}^{\infty} |a_n(g)|
\leq 2 \sum_{n=m+1}^{\infty} \exp \left(-\lambda^*\left(n\right)\right) 
\left\Vert g\right\Vert_{\A^{\lambda}\left(-1,1\right)}.
\end{aligned}
\end{align}
The sum may be bounded as in the proof of Theorem \ref{thm:periodic}. 

\emph{Part 4:} We again follow our general strategy. A complete orthonormal system of 
$\calU_m=\mathrm{span}\{x^0,\cdots,x^m\}\subset L^2(-1,1)$ is given by 
$\{\sqrt{j+1/2}\Legr_j:j=0,\dots,m\}$ with the  Legendre polynomials $\Legr_j$. As 
$\|p_j\|_{L^{\infty}(-1,1)}=1$, the supremum of the kernel 
$K(x,y)=\sum_{j=0}^m(j+1/2)\Legr_j(x)\Legr_j(y)$ is bounded by 
$\kappa_m=\sum_{j=0}^m(j+1/2)=(m+1)^2/2$. Hence, we choose 
$m(\delta):= \lfloor \sqrt{2/\delta}\rfloor-1$. 
\end{proof}

\subsection{Case $\manifold= \SS^2$}
Let $\SS^2:=\{x\in\Rset^3:|x|_2=1\}$ denote the unit sphere. 
Recall that if $\mathcal{P}_m$ denotes the space of polynomials 
in $x=(x_1,x_2,x_3)\in\Rset^3$ of degree $\leq m$, then the space of 
\textit{spherical harmonics} $\mathcal{H}_m$ is defined by 
\begin{align*}
\mathcal{H}_m&:= \left\lbrace\left. p_{\vert \SS^2}\right| p \in \mathcal{P}_m ~\text{harmonic and}~ p ~\text{is homogeneous of degree}~m\right\rbrace.
\end{align*}
Here harmonic means $\Delta p=0$ and homogeneous of degree $m$ that $p(rx)=r^mp(x)$ for 
all $r>0$ and $x\in\Rset^3$. 
We have the following decompositions as orthogonal direct sums with respect to $\langle \cdot, \cdot \rangle_{L^2(\SS^2)}$ (see e.g.~\cite[Cpt. 4]{stein1971introduction}): 
\begin{align}\label{sat10}
L^2(\SS^2) = \bigoplus_{l=0}^{\infty} \mathcal{H}_l, \qquad 
\mathcal{P}_m|_{\SS^2} = \bigoplus_{l=0}^m \mathcal{H}_l
\end{align}
One has $\dim \mathcal{H}_m = 2m+1$ and $\dim \mathcal{P}_m|_{\SS^2} = \sum_{l=0}^m (2l+1) = (m+1)^2.$ 
The orthogonal projections of $L^2(\SS^2)$ onto $\mathcal{H}_m$ and $\mathcal{P}_m|_{\SS^2}$ 
will be denoted by $Q_m$ and $\SHProj_m$, respectively. 
Spherical harmonics are closely related to Legendre polynomials $\Legr_m$. Choose any orthonormal basis $\left( Y_j\right)_{j=1}^{2m+1}$ of $\mathcal{H}_m$ with respect to $\langle \cdot, \cdot \rangle_{L^2(\SS^2)}$. The \textit{addition formula } of the spherical harmonics 
(see e.g.~\cite[Thm. 2]{muller1966spherical}) states that 
$\frac{2m+1}{4 \pi} \Legr_m \left(\langle x, y \rangle\right) = \sum_{j=1}^{2m+1} Y_j(x) Y_j(y)$, 
i.e.\
\begin{equation}\label{eq:proj_Hm}
(Q_mg)(x) = \frac{2m+1}{4 \pi}\int_{\SS^2}\Legr_m(\langle x,y\rangle) g(y)\,\mathrm{d}y, \qquad x \in\SS^2
\end{equation}
Let us introduce the averaging operator $M:C(\SS^2)\to C(\SS^2\times [-1,1])$ by 
\[
(Mg)(x,t):=\begin{cases}\frac{1}{2\pi\sqrt{1-t^2}}
\int_{\{y\in\SS^2|\langle y,x\rangle=t\}} g(y)\mathrm{d}y,&t\in (-1,1),\\
g(\pm x),& t=\pm 1. 
\end{cases}
\]
Note that $(Mg)(x,t)$ is the average of $f$ over a circle of radius $\sqrt{1-t^2}$ around $x$ 
and that $Mg$ is in fact continuous if $f$ is continuous. 
\begin{Def}
Let $\lambda:[0,\infty)\to \Rset\cup\{\infty\}$ be an increasing weight function. We define 
the space $\A^{\lambda}(\SS^2)$ by 
\begin{align*}
\A^{\lambda}(\SS^2)&:= \left\{g\in C(\SS^2)~\big|~ (Mg)(x,\cdot)\in\A^{\lambda}(-1,1)
\mbox{ for all }x\in\SS^2, \|g\|_{A^{\lambda}(\SS^2)}<\infty\right\},\\
\|g\|_{\A^{\lambda}(\SS^2)}&:= \sup_{x\in \SS^2} \|Mg(x,\cdot)\|_{\A^{\lambda}(-1,1)}.
\end{align*}
\end{Def}

\begin{thm}\label{thm:S2analytic}
Let $\lambda:[0,\infty)\to \Rset\cup\{\infty\}$ be non-decreasing with $B_{\lambda}>0$. 
\begin{enumerate}
\item $\A^{\lambda}(\SS^2)$ equipped with the norm $\|\cdot\|_{\A^{\lambda}(\SS^2)}$ 
is a Banach space. 
\item  For all $g\in\A^{\lambda}(\SS^2)$ we have
\begin{equation}\label{eq:proj_S2}
\|Q_mg\|_{L^{\infty}(\SS^2)}\leq 
\frac{2m+1}{4}\exp(-\lambda^*(m-1))\|g\|_{\A^{\lambda}(\SS^2)}.
\end{equation}
\item  There exist constants $c_{\lambda},d_{\lambda}>0$ and $m_0\in\{0,1,\dots\}$ such that 
\begin{equation}\label{eq:err_S2}
\|g-\SHProj_mg\|_{L^{\infty}(\SS^2)}
\leq (c_{\lambda}+md_{\lambda})\exp(-\lambda^*(m))\|g\|_{\A^{\lambda}(\SS^2)}
\end{equation}
for all $g\in\A^{\lambda}(\SS^2)$ and $m\geq m_0$. 
\item 
There exists $\delta_0>0$ such that the interpolation inequality \eqref{eq:general_interp_ieq} 
holds true with $\manifold =\SS^2$ and  
\[
\gamma(\delta) = 
\left(c_{\lambda}+\sqrt{\frac{4\pi}{\delta}}d_{\lambda}\right)
\exp \left(-\lambda^*\left( \left\lfloor \sqrt{\frac{4\pi}{\delta}} \right\rfloor-1\right)\right)
\]
for all $g\in\A^{\lambda}(\SS^2)$ and all $0<\delta\leq \delta_0$. 
\end{enumerate}
\end{thm}

\begin{proof}
\emph{Part 1:} Obviously $A^{\lambda}(\SS^2)$ is a normed space. To show completeness, let $(g_n)$ be 
a Cauchy-sequence in $A^{\lambda}(\SS^2)$. Then $(Mg_n)$ converges to some 
$G\in C(\SS^2\times [-1,1])$ with $\sup_{x\in\SS^2}\|G(x,\cdot)\|_{\A^{\lambda}(-1,1)}<\infty$ 
and $(g_n)$ converges uniformly to $G(\cdot,1)$. Since $M$ is continuous,  
it follows that $G= MG(\cdot,1)$. \\
\emph{Part 2:} Due to \eqref{eq:proj_Hm} and the formula 
$\int_{\SS^2}f(y)\,\mathrm{d}y = \int_{-1}^1 \frac{1}{\sqrt{1-t^2}}
\int_{\left\{y~|~\langle y,x\rangle=t\right\}}f(y)\,\mathrm{d}y\,\mathrm{d}t$
we have 
\begin{align*}
(Q_mg)(x) &= \frac{2m+1}{2}\int_{-1}^1 \Legr_m(t) (Mg)(x,t)\,\mathrm{d}t\\
&= \frac{2m+1}{2}\int_0^{\pi} \Legr_m(\cos(s)) (Mg)(x,\cos(s))\sin(s)\, \mathrm{d}s.
\end{align*}
With the mapping $\theta(w):=\frac{1}{2}(w+\frac{1}{w})$ from the proof of Theorem~\ref{thm:interval} 
and the substitution $w=\exp\left(\textup{i}s\right)$ we obtain
\begin{align*}
(Q_mg)(x) &= \frac{2m+1}{4}\Re\int_{\{w\in\Cset~|~|w|=1\}}\Legr_m\left(\theta(w)\right) (Mg)(x,\theta(w))\,\mathrm{d}w.
\end{align*}
We may deform the contour of integration $\{w|~|w|=1\}$ to any contour $\{w|~|w|=\exp(r)\}$ 
with $|r|<B_\lambda$. To estimate $Q_mg$, we have to 
estimate the growth of the Legendre polynomials on $\mathcal{E}_{r}$.
We use the identity 
$\Legr_m(z) = \frac{1}{\pi} \int_0^{\pi} \left( z+\sqrt{z^2-1} \cos \varphi\right)^m \,\mathrm d\varphi$
(see \cite[Cpt. 4]{szego1959orthogonal}) to find that 
\begin{align*}
\Legr_m\left( \frac{1}{2}\left( w+\frac{1}{w}\right)\right) 
&= \frac{1}{\pi} \int_0^{\pi} \left[ w \frac{1+ \cos \varphi}{2} + \frac{1}{w} \frac{1- \cos \varphi}{2}\right]^m \,\mathrm d\varphi\\
&= \frac{2}{\pi} \int_0^{\frac{\pi}{2}} \left[ w \cos^2 \psi + \frac{1}{w} \sin^2 \psi \right]^m \,\mathrm d\psi.
\end{align*}
Therefore, 
$\sup_{z\in\partial\mathcal{E}_r}|\Legr_m(z)|
=\sup_{|w|=r}\left|\Legr_m \left( \frac{1}{2} \left( w+\frac{1}{w}\right)\right)\right| \leq \exp(mr).$ 
It follows that 
\begin{align*}
|(Q_mg)(x)| &\leq \frac{2m+1}{4}\sup_{0\leq r< B_\lambda}
\int_{|w|=\exp(r)}|\Legr_m(\theta(w))|\;|(Mg)(x,\theta(w))|\,|\mathrm{d}w|\\
&\leq \frac{2m+1}{4}\sup_{r\geq 0}\left[\exp\left(r(m-1)-\lambda(r))\right)\right]\|g\|_{\A^{\lambda}(\SS^2)}\\
&= \frac{2m+1}{4}\exp(-\lambda^*(m-1))\|g\|_{\A^{\lambda}(\SS^2)}.
\end{align*}
\emph{Part 3:} As shown in the proof of Theorem \ref{thm:periodic} there exists $m_0\in\{0,1,\dots\}$ 
and $a>0$ such that $\lambda^*(m_0)>0$ and $\lambda^*(m+1)-\lambda^*(m)\geq a$ for all $m\geq m_0$. 
Using the identity $\sum_{j=0}^{\infty}r^j = 1/(1-r)$ for $|r|<1$ and its derivative 
$\sum_{m=1}^{\infty}mr^{m-1}=(1-r)^{-2}$ we obtain 
\begin{align*}
\|g-\SHProj g\|_{L^{\infty}(\SS^2)} 
&\leq \sum_{j=m+1}^{\infty} \|Q_mg\|_{L^{\infty}(\SS^2)}\\
&\leq \sum_{j=m+1}^{\infty}\frac{2j+1}{4}\exp(-\lambda^*(j-1))\|g\|_{\A^{\lambda}(\SS^2)}\\
&\leq \exp(-\lambda^*(m))\sum_{j=m+1}^{\infty}\frac{2j+1}{4}\exp(-\lambda^*(j-1)+\lambda^*(m))
\|g\|_{\A^{\lambda}(\SS^2)}\\
&\leq \exp(-\lambda^*(m))\|g\|_{\A^{\lambda}(\SS^2)}\sum_{l=0}^{\infty}
\left(\frac{l+1}{2}+\frac{2m+1}{4}\right)\exp(-al)\\
&= \exp(-\lambda^*(m))\|g\|_{\A^{\lambda}(\SS^2)}
\left(\frac{1}{2(1-\exp(-a))^2}+\frac{2m+1}{4-4\exp(-a)}\right)
\end{align*}
This shows \eqref{eq:err_S2}.

\emph{Part 4:}
By \eqref{sat10} and \eqref{eq:proj_Hm} the kernel of the $L^2$-orthogonal projection 
onto $\mathcal{P}_m|_{\SS^2}$ is given by 
$K_m(x,y)= \frac{1}{4\pi} \sum_{l=0}^m (2l+1) \Legr_l (\langle x, y \rangle)$. 
As $\|\Legr_l\|_{L^{\infty}(-1,1)}=1$, the supremum norm of $K_m$ is bounded by 
$\kappa_m=\frac{1}{4 \pi}\sum_{l=0}^m(2l+1)=\frac{1}{4 \pi}(m+1)^2$. 
Therefore, we choose $m(\delta)=\lfloor \sqrt{4\pi/\delta}\rfloor-1$. 
\end{proof}

\section{Examples}\label{sec:ex}

In this section we show how our general techniques can be applied to 
practical examples. As we have seen in \S~2, the key ingredients are the 
variational source condition \eqref{eq:vsc} and the smoothness assumption 
\eqref{Condition3}. As the variational source condition can be validated using 
Proposition~\ref{prop:vsc_valid}, we will focus on verifying \eqref{Condition3}. 
This is done exemplary for the forward operators connected to the periodic 
backwards heat equation and an inverse problem in satellite gradiometry on 
$\SS^2$. Similar techniques should apply to many other exponentially 
ill-posed problems, e.g.\ inverse scattering problems.

\subsection{Backwards heat equation}

Let $f \in L^2(-\pi, \pi).$ Let $u:(-\pi, \pi)\times [0, \infty) \rightarrow \Rset$ be a solution of the periodic heat equation
\begin{align}\label{heatequation}
\begin{cases} \frac{\partial u}{\partial t}(x,t) = \frac{\partial^2u}{\partial x^2}(x,t) &\mbox{if}~ (x,t) \in (-\pi, \pi) \times (0, \infty)\\
u(x,0)=f(x) &\mbox{if}~ x \in (-\pi, \pi) \hspace{0.4cm} \mbox{(initial condition)}\\
u(- \pi, t)= u (\pi, t) &\mbox{if}~ t \in (0, \infty)\hspace{0.4cm} \mbox{(boundary condition)}
\end{cases}
\end{align}
It describes heat propagation on a circle parameterized by $x$ as angular variable. The initial boundary value problem \eqref{heatequation} has a unique solution given by 
\begin{equation}\label{eq:heat_eq_sol}
u(x,t) = \sum\limits_{n \in \Zset} \exp\left(-n^2t\right)~ \hat{f}(n) \exp\left(\textup{i} n x\right), \qquad (x,t)\in [-\pi, \pi ] \times [0, \infty)
\end{equation}
where $\hat{f}(n)$ are the Fourier coefficients of $f$. 

Fix a time $\bar{t}>0$. In the following we will study the forward operator $F$ associated to the backwards heat equation, which consists in determining the initial heat distribution $f$ from the measurement of the heat distribution $u(x, \bar{t})$ at time $\bar{t}$. This problem is known to be exponentially ill-posed, which can readily be seen from the decay of the singular values $\exp\left(-n^2 t\right)$ in \eqref{eq:heat_eq_sol}.

Let us therefore define $\OpL:L^2(-\pi, \pi) \rightarrow  L^1(-\pi, \pi)$ by 
\begin{align}\label{Equation25}
(\OpL f)(x) := \sum\limits_{n \in \Zset} \exp\left(-n^2 \bar{t}\right) ~\hat{f}(n) \exp\left(\textup{i} n x\right), \qquad x \in \left[-\pi,\pi\right]
\end{align}
and set $g:= u(\cdot, \bar{t})$.

\begin{lem}\label{lem:periodic_heat_eq}
Let $\OpL$ as in \eqref{Equation25}, set $\X:=L^2 \left(-\pi,\pi\right)$ and $\R\left(f\right) = \left\Vert f \right\Vert_{L^2 \left(-\pi,\pi\right)}^2$.
\begin{enumerate}
\item Then $\OpL : \X \to \A^\lambda \left(\Rset / 2\pi \Zset\right)$ is bounded with 
\[
\lambda\left(r\right)= \frac{r^2}{4 \bar t}, \qquad r \in \left[0,\infty\right).
\]
Furthermore Assumption \ref{ass:interp_ineq} holds true with $q = 2$ and
\[
\gamma(\delta)\leq \frac{8\max\left\{1,\bar t^{-1/2}\right\}}{\exp\left(3 \bar t\right)-1} 
\exp \left(- \left\lfloor \frac{\pi}{\delta} - \frac{1}{2} \right\rfloor^2 \bar{t} \right).
\]
\item If $\udag \in H^p(-\pi, \pi)$, then Assumption \ref{ass:vsc} holds true with any bounded $D(F)$, $\manifold = (-\pi,\pi)$, $r = 1$, some $\beta, \beta'>0$ and
\[
\varphi\left(\tau\right) = \beta' \varphi_{p} \left(\tau\right), \qquad \tau > 0.
\]
\end{enumerate}
\end{lem}
\begin{proof}
\emph{Part 1:} First we estimate $\|\OpL\|_{\X \rightarrow \A^\lambda \left(\Rset / 2 \pi \Zset\right)}$. For $x \in \Rset, t>0$ we will use the identity
\begin{equation}\label{eq:aux4}
\frac{1}{\sqrt{\pi}} \sum_{n \in \Zset} \exp \left( -(n+x)^2 t \right) \leq \frac{1}{\sqrt{\pi}} \sum_{n \in \Zset} \exp \left(-n^2 t\right) \leq \sqrt{2}\max\left\{1, \frac{1}{\sqrt{t}}\right\},
\end{equation} 
which follows from Poisson's summation formula (cf. \cite[Cpt. 5, Thm. 3.1]{stein2003princeton}). 

Applying Cauchy-Schwarz, $\sup_{\Im\left(z\right) = \pm r} |\exp(\textup{i} n z)| \leq \exp(|n| r)$ and \eqref{eq:aux4} with $x= \pm \frac{r}{ 2 \bar{t}}$ we find
\begin{align*}
\underset{\Im\left(z\right) =  \pm r}{\sup} |\OpL f(z)| &\leq \sum_{n \in \Zset} \exp(-n^2 \bar{t}) \exp(|n| r) \frac{1}{\sqrt{2 \pi}} \|f\|_{\X}\\
&= \frac{1}{\sqrt{ 2 \pi}} \exp(\lambda(r)) \|f\|_{\X} 
\sum\limits_{n \in \Zset} \exp \left( - \left( |n|-\frac{r}{2 \bar{t}} \right)^2 \bar{t}\right) 
\end{align*}
The sum on the right hand side can be bounded as follows:
\begin {align*}
\sum\limits_{n \in \Zset} \exp \left( - \left( |n|-\frac{r}{2 \bar{t}} \right)^2 \bar{t}\right) 
&\leq \sum\limits_{n \in \Zset} 
\left[\exp \left( - \left( n - \frac{r}{2 \bar{t}}\right)^2 \bar{t} \right) 
+  \exp \left( - \left( n+ \frac{r}{2 \bar{t}}\right)^2 \bar{t}\right)\right]\\
&\leq 2\sum\limits_{n \in \Zset} \exp(-n^2 \bar{t}) 
\leq 2 \sqrt{2\pi} \max\left\{1,\bar{t}^{-1/2}\right\}.
\end{align*}
This implies 
$\|\OpL\|_{\X \rightarrow \A^\lambda \left(\Rset / 2 \pi \Zset\right)} \leq  2 \max\left\{1,\bar{t}^{-1/2}\right\}$.
The assertions will now follow from Remark \ref{rem:smoothness}. The conjugate function of $\lambda$ is $\lambda^*(s) = s^2 \bar{t}$. 
By Theorem \ref{thm:periodic} we find that \eqref{condition2} holds true for all 
$0 < \delta \leq \delta_0,~\delta_0 >0$ sufficiently small with
\begin{align*}
\gamma(\delta) &= \frac{4}{\exp\left(3 \bar t\right) -1}
\exp \left(- \left\lfloor \frac{\pi}{\delta}- \frac{1}{2}\right\rfloor^2 \bar{t} \right)
\end{align*}
since $m_0=1$ and $a=\lambda^*(m_0+1)-\lambda^*(m_0)=3\bar t$
in the proof of Theorem \ref{thm:periodic}. This yields the claim.\\
\emph{Part 2:} One readily calculates that
\[
\varphi_p(\OpL^*\OpL)(w)(x) \sim \sum\limits_{n \in \Zset} \frac{1}{(\bar{t}n^2)^p} \hat{w}(n) \exp\left(\textup{i} n x\right),
\]
i.e.\ $\|\varphi_p(\OpL^*\OpL)w\|_{H^{2p}(-\pi, \pi)} \sim \|w\|_{L^2(-\pi, \pi)}$. Thus $\udag \in H^p(-\pi,\pi)$ is equivalent to $\udag = \varphi_{p/2}\left(\OpL^*\OpL\right)w$ with $w \in L^2(-\pi,\pi)$. Now Proposition~\ref{prop:vsc_valid} with $\manifold=(-\pi, \pi)$ yields the claim.
\end{proof}
\begin{thm}\label{Theorem5}
Let $\bar{t}>0$ and $\OpL$ be as in \eqref{Equation25} and suppose $\udag \in H^p(-\pi, \pi)$ with some $p>0$. Suppose furthermore that $\alpha = \alpha \left(\eta, \epsilon\right)$ is chosen such that \eqref{eq:PC_condition} holds true in the limit $\eta, \epsilon, \alpha \searrow 0$.

Then there exists $\eta_0 >0$ and a constant $C = C \left(\bar t, \OpL, \udag\right)$ such that for any noise function $\xi \in L^1(-\pi, \pi)$ fulfilling \eqref{eq:noisemodel} with $0 < \eta \leq \eta_0,~\epsilon \geq 0$, the following estimates are valid for Tikhonov regularization \eqref{eq:tik} with $\R\left(f\right) = \left\Vert f \right\Vert_{L^2 \left(-\pi,\pi\right)}^2$, $r = 1$:
\begin{align}
\|\ualdel - \udag\|^2_{L^2} &\leq C \left(\frac{\epsilon}{\alpha} 
+ \frac{\eta^2}{\alpha^2}\exp\left(-\frac{\pi^2}{8\eta^2}\bar t\right) 
+ \left(- \ln\left(\alpha\right)\right)^{-p}\right)(1+o(1)),\label{Inequality26}\\
\|\OpL (\ualdel)- \gdag\|_{L^1} &\leq C \left(\epsilon 
+ \frac{\eta^2}{\alpha}\exp\left(-\frac{\pi^2}{8\eta^2}\bar t\right) 
+ \alpha \left(- \ln \left(\alpha\right) \right)^{-p}\right)(1+o(1)),\label{Inequality27}
\end{align}
as $\eta, \epsilon, \alpha \searrow 0$.
The parameter $\alpha =\bar{\alpha}(\epsilon, \eta)$ can be chosen such that
\begin{align*}
\|\hat{f}_{\bar{\alpha}}-\udag\|^2_{L^2} &=\mathcal O \left(\max\left\{\left(-\ln\left(\epsilon\right)\right)^{-p},\eta^{2p} \right\}\right),\\
\|\OpL(\hat{f}_{\bar{\alpha}}) - \gdag\|_{L^1} &=\mathcal O \left(\max\left\{\epsilon, \exp\left(-\frac{\pi^2}{16 \eta^2} \bar t\right)\eta^{2p+1}\right\}\right)
\end{align*}
as $\epsilon, \eta \searrow 0$. 
\end{thm} 
\begin{proof}
As \eqref{eq:PC_condition} holds true, we know from Proposition~\ref{prop:vsc_valid} that \eqref{eq:vsc} can be applied with $f = \ualdel$. From Lemma \ref{lem:periodic_heat_eq} we know that we may apply Theorem \ref{Theorem1} with $\varphi\left(\tau\right) = \beta' \varphi_p \left(\tau\right)$ and
\begin{align*}
\gamma(\delta) &\leq  \frac{8\max\left\{1, \bar t^{-1/2}\right\}}{\exp\left(3 \bar t\right)-1} 
\exp \left(- \left\lfloor \frac{\pi}{\delta} - \frac{1}{2} \right\rfloor^2 \bar{t} \right) \\
& \leq 8 \frac{\max\left\{1,\bar t^{-1/2}\right\}}{\exp\left(3 \bar t\right)-1} 
\exp\left(-\frac{\pi^2}{\delta^2}\bar t\right) \left(1+o\left(1\right)\right), \qquad \delta \searrow 0.
\end{align*}
The function $\psi$ in \eqref{eq:defi_associate} is given by
\[
\psi\left(\alpha\right) = \beta'\left(\ln\left(\frac{1}{\beta' \alpha p}\right)\right)^{-p} \left(1+o\left(1\right)\right) \leq \left(\beta'\right)^2 p \left(- \ln\left(\alpha\right)\right)^{-p}\left(1+o\left(1\right)\right).
\]
Inserting these results into \eqref{Inequality4} and \eqref{Inequality5} with $q = q' = 2$ and $b=2$ yields \eqref{Inequality26} and \eqref{Inequality27}.

An optimal $\alpha$ can be chosen now depending on which of the first two terms in \eqref{Inequality26} dominates. If the first term is larger than the second, we choose $\alpha_1$ to balance the first and the last term in \eqref{Inequality26}, which yields
\begin{align*}
\frac{\epsilon}{\alpha_1} = \left(- \ln\left(\epsilon\right)\right)^{-p} \left( 1+o(1) \right), \qquad \alpha_1, \epsilon \to 0.
\end{align*}
On the other hand, if the second term is larger than the first, we choose $\alpha_2$ such that the second and third term equal, which gives
\begin{align*}
\left(-\ln\left(\alpha_2\right)\right)^{-p} = \bar{t}^{-p} \left( \frac{4 \eta }{\pi}\right)^{2p}\left( 1+o(1) \right), \qquad \alpha_2, \eta \to 0.
\end{align*}
Taking the maximum of both cases yields the proposed bounds. As mentioned in Remark \ref{rem:pc}, the condition \eqref{eq:PC_condition} is satisfied.
\end{proof}

\subsection{An inverse problem in satellite gradiometry}
Let us assume, the earth is described by the unit ball $B:=\{x\in\Rset^3:|x|\leq 1
\}$. 
In geophysics the value of the gravitational potential $u$ on the surface of the earth is of 
interest as it contains information about the interior of the earth. If the value $f$ of 
$u$ on $\partial B$ is given, $u$ is the solution to the exterior boundary value problem
\[
\begin{cases}
\Delta u= 0&\mbox{in }\Rset^3\setminus B,\\
u = f&\mbox{on }\partial B = \SS^2,\\
|u(x)| = \mathcal{O}(|x|^{-1})&\mbox{as }|x|\to \infty.
\end{cases}
\]
Using satellites it is possible to measure the second derivative of $u$ in radial direction 
$r=|x|$ at some distance $R>1$ from the earth, i.e.\ the rate of change of the gravitational 
force: 
\[
g= \frac{\partial^2u}{\partial r^2}  \qquad \text{on}~ R\SS^2
\]
The inverse problem of gradiometry then consists in estimating $f$ given $g$ (see 
\cite[Sec. 8.2]{h00} and references therein). Representing $u$ by the Poisson formula 
for the exterior of $B$, it can be seen (see e.g. \cite{FSS:97}) that the forward operator is given by 
\begin{align}\label{sat3}
\begin{aligned}
(\OpL f)(x)&= \frac{1}{|\SS^2|} \int_{\SS^2} \frac{\partial^2}{\partial R^2} \left\lbrace R^{-1} \frac{1-R^{-2}}{|R^{-1}x-y|^{3}}\right\rbrace f(y) \,\mathrm d y \\
&= \sum_{m=0}^{\infty}\frac{(m+1)(m+2)}{R^{m+3}}(Q_mf)(x)\quad x \in \SS^2
\end{aligned}
\end{align}
with the orthogonal projection $Q_m$ onto the spherical harmonics of order $m$ introduced in \eqref{eq:proj_Hm}. This representation of $T$ shows again that the problem to recover $f$ from $g$ is exponentially ill-posed.

\begin{lem}\label{lem:gradiometry}
Let $\OpL$ as in \eqref{sat3} with $R>1$, $\X=L^2(\SS^2)$ and $\R\left(f\right) = \left\Vert f \right\Vert^2_{L^2 \left(\SS^2\right)}$. 
\begin{enumerate}
\item $T$ is a bounded mapping from $L^2(\SS^2)$ to $\A^{\lambda}(\SS^2)$ for 
\[
\lambda(r) := \begin{cases}-4\ln \left(R-\exp(r)\right),&0\leq r< \ln(R), \\
\infty& \mbox{else.}\end{cases}
\]
\item Assumption \ref{ass:interp_ineq} holds true with
\begin{equation}\label{eq:gamma_grad}
\gamma(\delta) = C \delta^{-5/2}R^{-\sqrt{4\pi/\delta}-4}
\end{equation}
with $C$ independent of $R$ and $\delta$. 
\item If $\udag \in H^p\left(\SS^2\right)$, then Assumption \ref{ass:vsc} holds true with any 
bounded $D\left(F\right)$, $\manifold = \SS^2$, $r= 1$, some $\beta,\beta'>0$ and
\begin{align*}
\varphi\left(\tau\right) = \beta' \varphi_{2p} \left(\tau\right), \qquad \tau > 0.
\end{align*}
\end{enumerate}
\end{lem}
\begin{proof}
\emph{Part 1:} First note that for $g_m\in \mathcal{H}_m$ and $n\in\{0,1,\dots,\}$ we have 
\begin{align*}
\int_{-1}^1\Legr_n(t)(Mg_m)(x,t)\D t
&= \int_{-1}^1\frac{\Legr_n(t)}{2\pi\sqrt{1-t^2}} \int_{\langle x,y\rangle = t}g_m(y)\D y\D t\\
&= \int_{\SS^2}\Legr_n(\langle x,y\rangle) g_m(y) \,\mathrm d y\\
&= \frac{4\pi}{2m+1}(Q_ng_m)(x) \\
&= \delta_{n,m}\frac{4\pi}{2m+1}g_m(x).
\end{align*}
As $\{\Legr_m/\sqrt{2m+1}\}$ is an orthonormal basis of $L^2(-1,1)$, we find that 
\(
(Mg_m)(x,t) = \frac{p_m(t)}{4\pi}g_m(x). 
\)
Together with \eqref{sat3} we obtain 
\[
(MTf)(x,t) = \sum_{m=0}^{\infty} \frac{p_m(t)}{4\pi} \frac{(m+1)(m+2)}{R^{m+3}} (Q_mf)(x).
\]
Using the bound $|p_m(z)|\leq \exp(mr)$ for $z\in\partial\mathcal{E}_r$ 
from the proof of Theorem \ref{thm:S2analytic} and $\|Q_mf\|_{L^{\infty}(\SS^2)}
\leq (2m+1)\|f\|_{L^2(\SS^2)}$ we obtain 
\begin{align*}
\|Tf\|_{\A^{\lambda}(\SS^2)}
&\leq \sup_{r\geq 0}\sup_{z\in\partial\mathcal{E}_r} \sup_{x\in\SS^2}\exp(-\lambda(r))|(MTf)(x,z)|\\
&\leq \sup_{0\leq r< \ln(R)} (R-e^r)^4
\sum_{m=0}^{\infty} \frac{p_m(z)}{4\pi} \frac{(m+1)(m+2)}{R^{m+3}} \|Q_mf\|_{L^{\infty}(\SS^2)}\\
&\leq  \sup_{0\leq r< \ln(R)} \frac{(R-e^r)^4}{4\pi}
\sum_{m=0}^{\infty}\frac{(m+1)(m+2)(2m+1)e^{rm}}{R^{m+3}}\|f\|_{L^2(\SS^2)}\\
&\leq CR\|f\|_{L^2}.
\end{align*}
The last inequality follows by evaluating derivatives of the geometric series, and $C$ is a 
generic constant independent of $R$.\\
\emph{Part 2:} 
The supremum of $r\mapsto sr-\lambda(r)$ is attained if $e^r=sR/(4+s)$, so 
\begin{align*}
\exp\left(-\lambda^*(s)\right) = \left(\frac{4+s}{4R}\right)^4 \left(\frac{4+s}{sR}\right)^s
\leq Cs^4 R^{-s-4}.
\end{align*}
Now \eqref{eq:gamma_grad} follows from Theorem \ref{thm:S2analytic}. \\
\emph{Part 3:} The proof is similar to the heat equation case using \cite[Prop. 16]{h00}. There it has been shown that
\[
\udag  = \varphi_p \left(T^*T\right) \omega, \omega \in L^2 \left(\SS^2\right) \qquad \Leftrightarrow \qquad \udag \in H^p \left(\SS^2\right).
\]
Thus the spectral source condition is satisfied by assumption, and it follows from Corollary \ref{cor:vsc} that \eqref{eq:vsc} holds true with $r = 1$, and $\varphi = \beta'\varphi_{2p}$ for any bounded $D\left(T\right)$. 
\end{proof}

\begin{thm}\label{Theorem8}
Let $R>1$, $\OpL$ as in \eqref{sat3} and $\udag \in H^p\left(\SS^2\right)$ with some $p>0$. Suppose furthermore that $\alpha = \alpha \left(\eta, \epsilon\right)$ is chosen such that \eqref{eq:PC_condition} holds true in the limit $\eta, \epsilon, \alpha \searrow 0$.

Then there exists $\eta_0 >0$ and a constant $C = C \left(R, \OpL, \udag\right)$ such that for any noise function $\xi \in L^1(\SS^2)$ fulfilling \eqref{eq:noisemodel} with $0 < \eta \leq \eta_0,~\epsilon \geq 0$, the following estimates are valid for Tikhonov regularization \eqref{eq:tik} with $\R\left(f\right) = \left\Vert f \right\Vert_{L^2 \left(\SS^2\right)}^2$, $r = 1$:
\begin{align*}
\|\hat{f}_{\alpha} - f^{\dagger}\|^2_{L^2(\SS^2)} &\leq C \left( \frac{\epsilon}{\alpha} + \frac{1}{\alpha^2}\frac{1}{R^6} \eta^{-3} R^{-2\sqrt{\pi/\eta}} +  \left(-\ln \alpha\right)^{-2p} \right)\left(1+o\left(1\right)\right),\\
\|\OpL(\hat{f}_{\alpha}) - \OpL(f^{\dagger})\|_{L^1(\SS^2)} &\leq C \left( \epsilon + \frac{1}{\alpha}\frac{1}{R^6} \eta^{-3}R^{-2\sqrt{\pi/\eta}} +  \alpha \left(-\ln \alpha\right)^{-2p}\right)\left(1+o\left(1\right)\right),
\end{align*}
as $\eta, \epsilon, \alpha \searrow 0$.
The parameter $\alpha = \bar \alpha (\epsilon, \eta)$ can be chosen such that
\begin{align*}
\|\hat{f}_{\bar \alpha} - f^{\dagger}\|^2_{L^2(\SS^2)} &= \mathcal{O} \left(\max \left\lbrace \left(- \ln \epsilon\right)^{-2p}, \eta^{p} \right\rbrace\right)\\
\|\OpL(\hat{f}_{\bar \alpha}) - \OpL(f^{\dagger})\|_{L^1(\SS^2)} &= \mathcal{O} 
\left(\max \left\lbrace \epsilon, \eta^{\frac{2p-3}{2}} R^{-\sqrt{\pi/\eta}}\right\rbrace\right)
\end{align*}
as $\epsilon, \eta \searrow 0$. 
\end{thm}
\begin{proof}
As \eqref{eq:PC_condition} holds true, we know from Proposition~\ref{prop:vsc_valid} that \eqref{eq:vsc} can be applied with $f = \ualdel$. From Lemma \ref{lem:gradiometry} we know that we may apply Theorem \ref{Theorem1} with $\varphi\left(\tau\right) = \beta' \varphi_{2p} \left(c\tau\right)$ and $\gamma$ as in \eqref{eq:gamma_grad}. This yields the claimed error estimates.

To obtain the asserted convergence rates, we may again distinguish if the $\epsilon$ term or the $\eta$ term dominates. Then we balance the dominating term with the pure $\alpha$ term to obtain a choice of $\alpha$ satisfying \eqref{eq:PC_condition}. The maximum of both cases yields the claim.
\end{proof}

\section{Conclusions}

We have extended a recent approach for inverse problems with impulsive noise to the case 
of infinitely smoothing forward operators using standard Sobolev smoothness assumptions for 
the solution. 
Remarkably, one obtains high order polynomial rates of convergence in the size  
$\eta$ of the corrupted domain, even though the underlying problem is exponentially ill-posed 
in the classical sense.
We examined two exponentially ill-posed problems arising in PDEs and showed that our  
analysis can be applied.

Our study gives rise to several further interesting questions: One concerns the extension 
of our analysis to more general domains in $\Rset^d$. A possible strategy could be the 
use of local averages as introduced here for the case of $\SS^2$. However, boundaries would 
cause technical difficulties, and the choice of approximating subspaces is not obvious. 
So far we have only studied a priori parameter choice rules, i.e. the solution smoothness characterized by $\varphi$ has to be known for choosing $\alpha$. The popular discrepancy principle does not apply in our context (as only $\left\Vert \xi \right\Vert_{L^1 \left(\mi\right)}$ is small, but $\mc$ is unknown), and thus other a posteriori parameter choice rules would be of interest for practical applications.
Other open questions include the optimality of the error bounds and an analysis of iterative 
regularization methods for nonlinear inverse problems with impulsive noise. 

\section*{Acknowledgement}

Financial support by the German Research Foundation DFG through subprojects A04 and C09 of CRC 755 is gratefully acknowledged. 

\bibliography{impulsive_noise}{}
\bibliographystyle{plain}
\end{document}